\documentclass[10pt]{article}
\usepackage{amsmath,amsthm,amsfonts,amssymb,amscd, amsxtra}
\usepackage[latin1]{inputenc}

\theoremstyle{plain}
\newtheorem{theorem}{Theorem}[section]
\newtheorem{lemma}{Lemma}[section]
\newtheorem{definition}{Definition}[section]
\newtheorem{corollary}{Corollary}[section]
\newtheorem{proposition}{Proposition}[section]

\newtheorem{remark}{Remark}[section]

\newtheorem{method}{Method}[section]
\newcommand{\beq}{\begin{equation}}
\newcommand{\eeq}{\end{equation}}
\newcommand{\beqa}{\begin{eqnarray}}
\newcommand{\eeqa}{\end{eqnarray}}
\newcommand{\beqas}{\begin{eqnarray*}}
\newcommand{\eeqas}{\end{eqnarray*}}
\def\min{\operatorname{min}}
\def\max{\operatorname{max}}

\makeatletter
\renewcommand*{\@biblabel}[1]{\hfill#1.}
\makeatother

\begin{document}

\title{The self regulation problem as an inexact steepest descent method for multicriteria optimization}

\author{
G. C. Bento\thanks{The author was supported in part by CNPq Grant 473756/2009-9 and PROCAD/NF. IME-Universidade Federal de Goi\'as,
Goi\^ania-GO 74001-970, BR (Email: {\tt glaydston@mat.ufg.br})}
\and
Da Cruz Neto, J. X.
\thanks{DM, Universidade Federal do Piau\'{\i},
Teresina, PI 64049-500, BR (Email: {\tt jxavier@ufpi.br}). This author
was partially supported by CNPq GRANT 302011/2011-0 and PRONEX--Optimization(FAPERJ/CNPq)}
\and
P. R. Oliveira \thanks{This author was supported in part by CNPq. COPPE/Sistemas-Universidade Federal do Rio de Janeiro,
Rio de Janeiro, RJ 21945-970, BR (Email: {\tt poliveir@cos.ufrj.br}).}
\and
A. Soubeyran \thanks{GREQAM-AMSE, Aix-Marseille University,
antoine.soubeyran@gmail.com}
}

\date{}
\maketitle
%\centerline{Communicated by Dinh The Luc}
\vspace{.2cm}

\noindent
{\bf Abstract}
In this paper, we study an inexact steepest descent method, with Armijo's rule, for multicriteria optimization. The sequence generated by the method is guaranteed to be well-defined. Assuming quasi-convexity of the multicriteria function we prove full convergence of the sequence to a critical Pareto point.  As an application, this paper offers a model of self regulation in Psychology, using a recent variational rationality approach.


%\noindent{\bf AMS Classification.} ...

%%%%%%%%%%%%%%%%
\section{Introduction}
The   steepest descent method with Armijo's rule for real  continuously differentiable optimization problem (see, for instance, Burachik et al. \cite{Bur1995}),   generates a sequence such that  any accumulation point of it, if any, is critical for the objective function.  This was later generalized for multicriteria optimization by Fliege and Svaiter \cite{Benar2000}, namely, whenever the objective function is a vectorial function.  The full convergence result for real optimization  problem  was assured when the solution set of the problem  be non-empty and the objective function is  convex (see Burachik et al. \cite{Bur1995}) or, more generally, a quasi-convex function (see Kiwiel and Murty \cite{Kiwiel1996}, Bello Cruz and Luc\^ambio P\'erez \cite{BeloCruz2010}). This result has been generalized for convex vectorial optimization by Gra\~na Drummond and Svaiter \cite{Benar2005} (see also Gra\~na Drummond and Iusem~\cite{Iusem2004}, Fukuda and Gra\~na Drummond \cite{Fukuda2011}) and, in the quasi-convex case, for multicriteria optimization by Bento et al.~\cite{Bento2012} (see also Bello Cruz et al.\cite{BeloCruz2011}). For extensions of other scalar optimization methods to the vectorial setting see, for instance, \cite{Benar2005-1, Benar2009, Mordukhovich2010} and references therein.

As far as we know, Bento et al. \cite{Bento2012} presented the first result of full convergence of the exact steepest descent method, with Armijo's rule, for quasi-convex multicriteria optimization, which includes contributions within Euclidean and Riemannian context; see also Bello Cruz et al. \cite{BeloCruz2011}. In the present paper, we study the method proposed by Fliege and Svaiter~\cite{Benar2000}, which is the inexact version of the method presented in \cite{Bento2012}. In this method are admitted relative errors on the search directions, more precisely, an approximation of the exact search direction is computed at each iteration. In this paper, we proved full  convergence of the sequence generated by this inexact method to a critical Pareto point associated to quasi-convex multicriteria optimization problems. In particular, we proved full convergence of the sequence to a weak Pareto optimal point in the case that the objective function is pseudo-convex.

The organization of our paper is as follows: In Section~\ref{sec1}, we present the self
regulation problem in the context of Psychology. In Section~\ref{sec2}, the multicriteria problem, the first order optimality condition for it and some basic definitions are presented. In Section~\ref{sec3}, the inexact steepest  descent method for finding one solution of  multicriteria problems is stated and the well-definedness of the sequence generated for it is established. In Section~\ref{sec4}, a partial convergence result for continuous differentiability multicriteria optimization is presented without any additional assumption on the objective function. Moreover, assuming that the objective function be quasi-convex and the Riemannian manifold has non-negative curvature, a full convergence result is presented. Finally,  Section~\ref{sec5} offers a ``distal-proximal" model of self regulation in Psychology, using a recent variational rationality approach (\cite{Soubeyran2009, Soubeyran2010, Soubeyran2012a, Soubeyran2012b}) which modelizes behaviors as an approach, or avoidance, course poursuit between ``desired, or undesired enough" ends, and ``feasible enough" means.

\section{The Self Regulation Problem}\label{sec1}

In this section devoted to applications we direct the attention of the
reader to the very important ``multiple goals" self regulation problem in
Behavioral sciences. We show the strong link between: $i)$ our paper which
extends the steepest descent methods of Fliege and Svaiter \cite{Benar2000}
to the quasiconvex case in multicriteria optimization and $ii)$ the
``variational rationality" approach of the ``theories of change" of
Soubeyran \cite{Soubeyran2009, Soubeyran2010, Soubeyran2012a, Soubeyran2012b}.%
Change problems consider ``why, how, and when"\ it is worthwhile to move
from a bad or not so good situation $x\in X$ to a better one $y\in X$ (known
or unknown), the limit case of full rationality being an optimizing one, the
case of bounded rationality being a better one (in a lot of different
formulations, depending of the context). The variational rationality
approach examines two polar kinds of ``change problems", choice and
transformation problems: i) adaptive choice problems like the ``choosing the
context to choose" problem (the formation of consideration sets), and ii)
transformation problems like creation and destruction, invention,
innovation, the evolution of institutions, dynamics interactions (dynamic
games), health, behavioral, organizational and cultural changes, $\ldots$ in
Economics, Decision theory, Management, Psychology, Artificial Intelligence,
Philosophy, Sociology, Applied Mathematics (Variational Analysis,
Optimization and Variational Inequalities). In this ``variational context"
our present paper shows how setting joint distal and proximal goals greatly
help to reach a distal goal. It offers an ``aspiration driven local search
proximal algorithm". This variational approach emphasizes, each step of the
process, two main variational principles (among others, in more general
settings): a ``satisficing-but not too much sacrificing" principle and a
``worthwhile to change" principle. Because the state space of situations is
the Euclidian space $X=\mathbb{R}^{n}$, changes $u=$ $y-x$ from a given
situation $x$ to a hopefully better situation $y$ can be characterized by
their directions $v\in X$ and their depth (length) $t>0$. In this context
where $u=tv$ these two variational principles specialize to,

\begin{itemize}
\item[i)] the choice, each step, of a ``satisficing but not too much
sacrificing" direction (a directional ``satisficing- but not too much
sacrificing" principle);

\item[ii)] the choice, each step, of a ``worthwhile change" (a ``worthwhile
to change" step length principle).
\end{itemize}

\subsection{Self regulation problems}

Self regulation considers the systematic activities (efforts) made to direct
thoughts, feelings, and actions, towards the attainment of one's goals
(Zimmerman \cite{Zimmerman2000}). A goal is a conscious or unconscious
mental representation of some future end (to approach or to avoid), more or
less distant, abstract (concrete), vague (precise), desirable and feasible.

Goals can be more or less desirable and more or less feasible. Related
to the desirability aspects are conscious or not, vague or concrete, distal
or proximal, long term or short term, extrinsic or intrinsic, set by others or oneself, individual or collective, learning or performance oriented
contents, high or low in commitment \ldots. Related to the feasibility aspects
are importance, priority, urgency, direction, intensity, difficulty,
measurability,

Self regulation have two aspects. The positive side of self regulation
considers purposive processes where agents engage in goal-directed actions. It examines goal setting, goal striving
and goal pursuit processes.

- Goal setting is the mental process of moving from the consideration of
distal goals to the formation of more proximal goals. Distal goals are
desired future ends (visions, imaginated desired futures), either promotion
aspirations \ (like ideals, fantasies, dreams, wishes, hopes and challenges)
or prevention aspirations (like oughts and obligations). They represent
desirable but quite irrealistic distal and vague ends (higher order goals).
Proximal goals can be wants, intentions, task goals, i.e much more feasible
but less desirable intermediate ends (sub goals).

- Goal striving ( goal implementation) examines the transition phase between
setting a distal goal and reaching it. 

- Goal pursuit (goal revision) focuses on the final phase, after reaching
the given goal or failing to reach it. It examines the role of feedbacks
(self evaluations of successes and failures, including the revision of
causal attributions and self efficacy beliefs, see Tolli and Schmidt \cite{Tolli2008}) in order to revise goals. 

The negative side of self regulation considers what an agent must refrain to
do instead of what he must do to set and attain some given goal. This
negative aspect of self regulation is named self control (overriding of one
action tendency in order to attain an other goal). It considers self
regulation failures like lack of vision, the inability to transform
irrealistic aspirations into intentions and realistic proximal goals
(preparation to action problems), procrastination and inertia (starting
problems), interruptions, distractions, temptations, lack of feedbacks, lack
of interest, perseverance and motivation (on the track problems) and goal
disengagement (ends problems).

Our paper considers only the positive aspect of self regulation. It focuses
on proximal goal setting activities, examines some aspects of goal revision
activities, and renounce to consider goal striving activities.

\subsection{Setting proximal goals}

\noindent\textbf{The Michael Jordan ``step by step" principle:} The famous basketball player Michael Jordan wrote the following about goal
setting in his book (Jordan and Miller \cite{Jordan1994}), \textquotedblleft
I approach everything step by step \ldots . I had always set short-term
goals. As I look back, each one of the steps or successes led to the next
one. When I got cut from the varsity team as a sophomore in high school, I
learned something. I knew I never wanted to feel that bad again \ldots . So
I set a goal of becoming a starter on the varsity. That's what I focused on
all summer. When I worked on my game, that's what I thought about. When it
happened, I set another goal, a reasonable, manageable goal that I could
realistically achieve if I worked hard enough \ldots . I guess I approached
it with the end in mind. I knew exactly where I wanted to go, and I focused
on getting there. As I reached those goals, they built on one another. I
gained a little confidence every time I came through \ldots .

\noindent\textbf{Goal hierarchies and goal proximity: the Bandura \ dual
``proximal-distal" self regulation principle:} Bandura \cite{Bandura1997} argued that people possess multiple systems of
goals, hierarchically arranged from proximal goals to extreme distal goals.
Goal proximity defines ``how far goals are conceptualized into the future".
A goal hierarchy interconnects at least three levels of goals: peak goals
(higher order goals, like visions, dreams, fantasies, aspirations, ideals,
wishes, hopes), distal goals (challenges), and task goals\ldots. A subset of task goals can be subordinate to distal goals which can
be subordinate to peak goals. Hence, the proximal goal distinction is
relative to the interconnected network of goals, other goal's providing the
temporal context). The main point to be emphasized is that distal goals and
proximal goals serve different and complementary conative functions
(connected to cognition, affect and motivation) related to goal difficulty,
goal commitment, psychological distance\ldots.

\begin{itemize}
\item[i)] distal goals define desired ends (enduring aspirations) that
attract individuals;

\item[ii)] proximal goals regulate immediate conative functions, which
provide the ways to find and follow a path of step by step changes moving
from the initial situation to approach the desired end or avoid an
undesirable end. In this context it is important to distinguish task goals
and strategies. The former defines what is to be accomplished, and the later
defines how it is to be accomplished (Wood and Bandura \cite{Wood1989}).
\end{itemize}

%%%%%%%%%%%%%%%%%%%%%%%%%%%%%%%%%%%%%%%
\section{The Multicriteria Problem}\label{sec2}
%%%%%%%%%%%%%%%%%%%%%%%%%%%%%%%%%%%%%%%
In this section, we present the multicriteria problem, the first order optimality condition for it and some basic definitions.

Let $I:=\{1,\ldots,m\}$, ${\mathbb R}^{m}_{+}=\{x\in {\mathbb R}^m: x_{i}\geq 0, j\in I\}$ and ${\mathbb R}^{m}_{++}=\{x\in {\mathbb R}^m: x_{j}> 0, j\in I \}$.
For $x, \, y \in {\mathbb R}^{m}_{+}$,
$y\succeq x$ (or $x \preceq y$) means that $y-x \in {\mathbb R}^{m}_{+}$
and $y\succ x$ (or $x \prec y$) means that
$y-x \in {\mathbb R}^{m}_{++}$.

Given a continuously differentiable vector function $F:{\mathbb R}^n\to {\mathbb R}^m$, we consider the problem of finding a {\it optimum Pareto  point} of F, i.e., a point $x^*\in {\mathbb R}^n$ such that there exists no other $x\in {\mathbb R}^n$ with $F(x)\preceq F(x^*)$ and $F(x) \neq F(x^*)$.  We denote this unconstrained problem as
\begin{equation} \label{eq:mp}
\min_{x\in {\mathbb R}^n} F(x).
\end{equation}
Let $F$ be given by $F(x):=\left(f_1(x), \ldots, f_m(x)\right)$. We denote the jacobian of $F$ by  
\[
J F(x):=\left(\nabla f_1(x), \ldots, \nabla f_m(x) \right), \qquad x\in {\mathbb R}^n,
\]
and the image of the jacobian of $F$ at a point $x\in {\mathbb R}^n$  by
\[
\mbox{Im} (J F(x)):=\left\{ J F(x)v=(\langle \nabla f_1(x),v\rangle, \ldots, \langle \nabla  f_m(x),v\rangle) : v\in \mathbb{R}^n \right\}.
\]
Using the above equality, the first-order optimality condition for Problem \ref{eq:mp} (see, for instance, \cite{Benar2000}) is stated as 
\begin{equation} \label{eq:oc} 
x\in {\mathbb R}^n, \qquad \mbox{Im}(JF(x))\cap (-{\mathbb R}^{m}_{++})=\emptyset.
\end{equation}
Note that the  condition in (\ref{eq:oc})  generalizes, to multicriteria optimization, the classical
condition ``gradient equals zero"  for the real-valued case.

In general, (\ref{eq:oc}) is necessary, but not sufficient, for optimality. A point of $\mathbb{R}^n$ satisfying (\ref{eq:oc}) is called {\it critical Pareto point}.

%%%%%%%%%%%%%%%%%%%%%%%%%%%%%%%%%%%%%%%%%%%%%%%%%%%%%%%%%%%%%%%%%%%
\section{Inexact Steepest Descent Methods for Multicriteria Problems}\label{sec3}
%%%%%%%%%%%%%%%%%%%%%%%%%%%%%%%%%%%%%%%%%%%%%%%%%%%%%%%%%%%%%%%%%%%
In this section, we state the inexact steepest  descent methods for solving  multicriteria problems admitting relative errors in the  search directions, more precisely, an approximation of the exact search direction is computed at each iteration, as considered in \cite{Benar2000, Benar2005, Fukuda2011}.

Let $x\in {\mathbb R}^n$ be a point which  is not critical Pareto point. Then there exists a direction $v\in \mathbb{R}^n$ satisfying
\[
J F(x)v \in -{\mathbb R}^{m}_{++},
\]
that is,  $J F(x)v \prec 0$. In this case, $v$ is called a {\it descent direction} for $F$ at $x$. 

For each $x\in {\mathbb R}^n$,  we consider the following unconstrained optimization problem in $\mathbb{R}^n$
\begin{equation} \label{eq:dsdd}
\mathop{\min}_{v\in \mathbb{R}^n} \; \left\{\mbox{max}_{i\in I}\langle\nabla f_i(x),v\rangle+ (1/2)\|v\|^{2}\right\}, \quad \qquad I:=\{1,\ldots,m\}.
\end{equation}

\begin{lemma}\label{dir:md1}$ $
The following statements hold:
\begin{itemize}
\item[i)] The unconstrained optimization problem in (\ref{eq:dsdd}) has only one solution. Moreover, the vector $v$ is the solution of the problem in (\ref{eq:dsdd}) if and only if there exists  $\alpha_i\geq 0$, $i\in I(x,v)$, such that 
\[
v=-\sum\limits_{i\in I(x,v)}\alpha_i\nabla f_i(x), \qquad \sum\limits_{i\in I(x,v)}\alpha_i=1,
\]
where $I(x,v):=\{i\in I: \langle\nabla f_i(x),v \rangle=\max_{i\in I}\langle\nabla f_i(x),v \rangle\}$;
\item[ii)] If $x$ is critical Pareto point of $F$ and $v$ denotes the solution of the problem in (\ref{eq:dsdd}), then $v=0$ and the optimal value associated to $v$ is equal to zero;
\item[iii)] If $x\in {\mathbb R}^n$ is not a critical Pareto point of $F$ and  $v$ is the solution of the problem in (\ref{eq:dsdd}), then $v\neq 0$ and
\[
\max_{i\in I}\langle\nabla f_i(x),v\rangle + (1/2)\|v\|^{2}<0.
\]
In particular, $v$ is a descent direction for $F$ at $x$.
\end{itemize}
\end{lemma}
\begin{proof}
The proof of the item  $i$ can be found in \cite{Bento2012}. For the proof of the remaining items, see, for example, \cite{Benar2000}.
\end{proof}

\begin{remark}\label{remark2012-10}
From the item $i$ of Lemma~\ref{dir:md1} we note that the solution of the minimization problem (\ref{eq:dsdd}) is of the form:
\[
v=-JF(x)^{t}w, \quad w=(\alpha_1,\ldots,\alpha_m)\in\mathbb{R}^m_+, \quad \|w\|_1=1\quad (\mbox{sum norm in}\; \mathbb{R}^m),
\]
with $\alpha_i=0$ for $i\in I\setminus I(x,v)$. In other words, if $S:=\{e_i\in \mathbb{R}^m:\; i\in I\}$ (set of the elements of the canonical base of Euclidean space $\mathbb{R}^m$), then $w$ is an element of the convex hull of $S(x,v)$, where
\begin{equation}\label{eq:2012-200}
S(x,v):=\{\bar{u}\in S:\; \langle\bar{u}, JF(x)v \rangle=\max_{u\in S}\langle u,JF(x)v \rangle\}.
\end{equation}
Note that the minimization problem (\ref{eq:dsdd}) may be rewritten as follows:
\[
\mathop{\min}_{v\in \mathbb{R}^n} \; \left\{\max_{u\in S}\langle u, JF(x)v\rangle+ (1/2)\|v\|^{2}\right\}=
\mathop{\min}_{v\in \mathbb{R}^n} \; \left\{\max_{u\in  S}\langle JF(x)^tu, v\rangle+ (1/2)\|v\|^{2}\right\}.
\]

\end{remark}
In view of the previous lemma and (\ref{eq:dsdd}), we define the steepest descent direction function for $F$ as follows.
\begin{definition}\label{def:sddf1}
The steepest descent direction  function for $F$ is defined as
\[
\mathbb{R}^n\ni x\longmapsto v(x):=\mbox{argmin}_{v\in \mathbb{R}^n} \left\{ \max_{i\in I}\left\langle\nabla f_i(x),v\right\rangle + (1/2)\|v\|^{2}\right\}\in \mathbb{R}^n.
\]
\end{definition}
\begin{remark}
This definition was proposed in \emph{\cite{Benar2000}}. Note that, from the item $i$ of \emph{Lemma \ref{dir:md1}} it follows that the steepest descent direction for vector functions becomes the steepest descent direction when $m=1$. 
\end{remark}

The optimal value associated to $v(x)$ will be denoted by $\alpha(x)$. Note that the function 
\[
\mathbb{R}^n\ni x\longmapsto \max_{i\in I}\left\langle\nabla f_i(x),v\right\rangle + (1/2)\|v\|^{2}\in\mathbb{R},
\]
is strongly convex with modulus $1/2$ and 
\[
0\in\partial \left(\max_{i\in I}\langle \nabla f_i(x),\;.\;\rangle+1/2\|.\|^2\right)(v(x)).
\]
So, for all $v\in \mathbb{R}^n$,
\begin{equation}\label{eq:2012}
\max_{i\in I}\left\langle\nabla f_i(x),v\right\rangle + (1/2)\|v\|^{2}-\alpha(x)\geq 1/2\|v-v(x)\|^2.
\end{equation}

\begin{lemma}\label{lemma2012-1}
The steepest descent direction function for $F$, $\mathbb{R}^n\ni x\mapsto v(x) \in \mathbb{R}^n$, is continuous. In particular, the function $\mathbb{R}^n\ni x\mapsto\alpha(x)\in\mathbb{R}$ is also continuous.
\end{lemma}
\begin{proof}
See \cite{Bento2012} for the proof of the first part. The second part is a immediate consequence of the first.
\end{proof}

\begin{definition}\label{def:2012}
Let $\sigma\in [0,1)$. A vector $v\in \mathbb{R}^n$ is say be a $\sigma-$approximate steepest descent direction at $x$ for $F$ if
\[
\max_{1\leq i\leq m}\langle\nabla f_i(x),v \rangle+1/2\|v\|^2\leq(1-\sigma)\alpha(x).
\]
\end{definition}
Note that the exact steepest descent direction at $x$ is a $\sigma$-approximate steepest descent direction for $F$ with $\sigma=0$. As a immediate consequence of Lemma~\ref{dir:md1} together with last definition, it is possible to prove the following:
\begin{lemma}\label{lemma2012}
Given $x\in {\mathbb R}^n$,
\begin{itemize}
\item[a)] $v=0$ is a $\sigma$-approximate steepest descent direction at $x$ if, only if, $x$ is a critical Pareto point;
\item[b)] if $x$ is not a critical Pareto point and $v$ is a $\sigma$-approximate steepest descent direction at $x$, then $v$ is a descent direction for $F$.
\end{itemize}
\end{lemma}

Next lemma establishes the degree of proximity between an approximate direction $v$ and the exact direction $v(x)$, in terms of the optimal value $\alpha(x)$.
\begin{lemma}\label{lemma2012-3}
Let $\sigma\in [0,1)$. If $v\in \mathbb{R}^n$ is a $\sigma-$approximate steepest descent direction at $x$, then
\[
\|v-v(x)\|^2\leq 2\sigma|\alpha(x)|.
\]
\end{lemma}
\begin{proof}
The proof follows from (\ref{eq:2012}) combined with Definition~\ref{def:2012}. See \cite{Benar2005}.
\end{proof}
 
 A particular class of $\sigma$- approximate steepest descent directions for $F$ at $x$ is given by the directions $v\in \mathbb{R}^n$ which are {\it scalarization compatible}, i.e., such that there exists $\tilde{w}\in \mbox{conv} S$ with
 \begin{equation}\label{eq:2012-350}
 v=-JF(x)^t\tilde{w}.
 \end{equation}
Note that $\tilde{w}$ determines a scalar function $g(x):=\langle\tilde{w},F(x)\rangle$ whose steepest descent direction coincides with $v$, which justifies the name previously attributed to the direction $v$; see \cite{Benar2005} for a good discussion.

Next proposition establishes a sufficient condition for $v$, given as in (\ref{eq:2012-350}), to be a $\sigma$-approximate steepest descent direction for $F$ at $x$. 
\begin{proposition}
Let $\sigma\in [0,1)$ and $v$ as in (\ref{eq:2012-350}). If 
\[
\max_{i\in I}\langle \nabla f_i(x),v\rangle\leq -(1-\sigma/2)\|v\|^2,
\]
or equivalently,
\begin{equation}\label{}
\max_{u\in S}\langle JF(x)^tu,v\rangle\leq -(1-\sigma/2)\|v\|^2,
\end{equation}
then $v$ is a $\sigma$-approximate steepest descent direction for $F$ at $x$.
\end{proposition}
\begin{proof}
See \cite{Benar2005}.
\end{proof}

From Remark~\ref{remark2012-10} we note that, for each $x\in {\mathbb R}^n$, the steepest descent direction for $F$ at $x$, $v(x)$, is scalarization compatible. Next lemma tell us that $v(x)$ satisfies the sufficient condition of the last proposition with $\sigma=0$ and, hence, that such condition is natural.

\begin{lemma} The following statements hold:
\begin{itemize}
\item[i)] $\alpha(x)=-(1/2)\|v(x)\|^2$;
\item[ii)] $\max_{u\in S}\langle JF(x)^tu,v(x)\rangle=-\|v(x)\|^2$.
\end{itemize}
\end{lemma}
\begin{proof}
In order to prove the item $i$ note that  
\begin{equation}\label{eq:2012-301}
\alpha(x)=\max_{u\in  S}\langle JF(x)^tu, v(x)\rangle+ (1/2)\|v(x)\|^{2}.
\end{equation}
Moreover, from Remark~\ref{remark2012-10}, we have
\begin{equation}\label{eq:2012-302}
v(x)=-JF(x)^tw, \qquad  w\in\mbox{conv} S(v(x)),\qquad S(v(x)):=S(x,v(x)),
\end{equation}
where  \mbox{conv}$S(v(x))$ denotes the convex hull of $S(v(x))$. So, combining (\ref{eq:2012-301}) and (\ref{eq:2012-302}) with the definition of $S(v(x))$, we get
\[
\alpha(x)=\langle JF(x)^t\bar{u}, -JF(x)^tw\rangle+ (1/2)\|JF(x)^tw\|^{2}, \qquad \bar{u}\in S(v(x)).
\]
Hence, 
\[
\alpha(x)=\langle JF(x)^tw, -JF(x)^tw\rangle+ (1/2)\|JF(x)^tw\|^{2},
\]
from where it follows the item $i$. The item $ii$ is an immediate consequence of the item $i$ combined with (\ref{eq:2012-301}).
\end{proof}

The {\it inexact steepest descent method with the Armijo rule} for solving the unconstrained  optimization problem (\ref{eq:mp}) is as follows:

\begin{method}[Inexact steepest descent method with Armijo rule]\label{algor1}
\hfill \vspace{.5cm}

\noindent
{\sc Initialization.} Take $\beta\in (0,\,1)$ and $x^0\in {\mathbb R}^n$. Set $k=0$.\\
{\sc Stop criterion.} If $x^k$ is a critical Pareto point STOP. Otherwise.\\
{\sc Iterative Step.} Compute a $\sigma$- approximate steepest descent direction $v^{k}$ for $F$ at $x^k$ and the steplength $t_k \in ]0,1]$ as follows:
\begin{equation} \label{eq:sl}
 t_k := \max \left\{2^{-j} : j\in {\mathbb N},\, F \left(x^k+2^{-j}v^k)\right)\preceq  F(x^k)+ \beta 2^{-j}\,J F(x^k)v^k\right\},
\end{equation}
and set
\begin{equation} \label{eq:xk}
x^{k+1}:=x^k+t_{k}v^k,
\end{equation}
and GOTO {\sc Stop criterion}.
\end{method}
\begin{remark}
The previous method was proposed by Fliege and Svaiter \cite{Benar2000} and becomes the classical steepest descent method when $m=1$. Other variants of Method~\ref{algor1} can be found in \cite{Benar2005, Iusem2004, Fukuda2011}. 
\end{remark}
Next proposition ensures that the sequence generated by the Method~\ref{algor1} is well-defined.  
\begin{proposition} \label{p:wd}
The sequence $\{x^k\}$ generated by the steepest descent method with Armijo rule is well-defined.
\end{proposition}
\begin{proof}
The proof follows from the item $ii$ of Lemma~\ref{lemma2012} combined with the fact that $F$ is continuously differentiable. See \cite{Benar2000} for more details.
\end{proof}

\section{Convergence Analysis}\label{sec4}
In this section, we present a partial convergence result without any additional assumption on $F$ besides the continuous differentiability. In the sequel, assuming quasi-convexity of $F$ and following the ideas of \cite{Benar2005, Bento2012}, we extend the full convergence result presented in \cite{Benar2005} for quasi-convex multicriteria optimization. It can be immediately seen that, if Method~\ref{algor1} terminates after a finite number of iterations, then it terminates at a critical Pareto point. From now on,  we will assume that $\{x^k\}$, $\{v^k\}$ and $\{t_k\}$ are infinite sequences generated by Method~\ref{algor1}.

To simplify the notation, in what follows we will utilize the scalar function  $\varphi:\mathbb{R}^m\to\mathbb{R}$ defined as follows:
\[
\varphi(y)=\max_{i\in I}\langle y,e_i\rangle, \qquad I=\{ 1,  \ldots, m \},
\]
where $\{e_i\}\subset \mathbb{R}^m$ is the canonical base of the space $\mathbb{R}^m$. It is easy to see that the following properties of the function $\varphi$ hold:
\begin{equation}\label{eq:prop1}
\varphi(x+y)\leq\varphi(x)+\varphi(y),\qquad\varphi(tx)=t\varphi(x), \qquad x,y\in\mathbb{R}^m, \quad t\geq 0.
\end{equation}
\begin{equation}\label{eq:prop2}
x\preceq y\quad \Rightarrow \quad  \varphi(x)\leq\varphi (y),\qquad x,y\in\mathbb{R}^m.
\end{equation}

\subsection{Partial Convergence Result}

The following theorem shows that if $F$ is continuously  differentiable then the sequence of the functional values of the sequence $\{x^k\}$, $\{F(x^k)\}$,  is monotonously decreasing and the accumulation points of $\{x^k\}$ are critical Pareto points. The proof of the next theorem can be found partly in \cite{Benar2000} and \cite{Benar2005}. We chose to present a proof within this paper.

\begin{theorem}\label{resconv1} The following statements hold:
\begin{itemize}
	\item [i)] $\{F(x^k)\}$ is decreasing;
	\item [ii)] If $\{x^k\}$ has accumulation point, then $\{t_k^2\|v^k\|^2\}$ is a summable sequence and
\begin{equation}\label{eq:2012-100}
	\lim_{k\to+\infty}t_k\|v^k\|^2=0;
	\end{equation}
	\item [iii)] Each accumulation point of the sequence $\{x^k\}$, if any, is a critical Pareto point.
\end{itemize}
\end{theorem}
\begin{proof}
The iterative step in Method~\ref{algor1} implies that 
\begin{equation}\label{eq:dir:md3}
F(x^{k+1})\preceq F(x^k)+\beta t_kJ F(x^k)v^k, \qquad x^{k+1}=x^k+t_kv^k, \qquad k=0, 1, \ldots .
\end{equation}
Since $\{x^k\}$ is an infinite sequence, for all $k$, $x^k$ is not a critical Pareto point of $F$. Thus, the item $i$ follows from the item $ii$ of Lemma~\ref{lemma2012} combined with the last vector inequality.

Suppose now that $\{x^k\}$ has an accumulation point $\bar{x}\in {\mathbb R}^n$ and let $\{x^{k_s}\}$ be a subsequence of $\{x^k\}$ such that $\lim_{s \to +\infty}x^{k_s}=\bar{x}$. Since $F$ is continuous and $\lim_{s \to +\infty}x^{k_s}=\bar{x}$ we have $\lim_{s \to +\infty}F(x^{k_s})=F(\bar{x})$. So, taking into account that $\{F(x^k)\}$ is a decreasing sequence and has $F(\bar{x})$ as an accumulation point, it is easy to  conclude that the whole sequence $\{F(x^k)\}$ converges to $F(\bar{x})$. So, from the definition of the function $\varphi$, we conclude that $\{\varphi(F(x^k))\}$ converges to $\varphi(F(\bar{x}))$ and, in particular,
\begin{equation}\label{eq:2012-10}
\varphi(F(\bar{x}))\leq \varphi(F(x^k)), \qquad k=0,1,\ldots.
\end{equation}
From (\ref{eq:dir:md3}), (\ref{eq:prop1}), (\ref{eq:prop2}) and definition of $v^k$, we obtain 
\[
\varphi(F(x^{k+1}))\leq\varphi(F(x^k))+\beta t_k\left((1-\sigma)\alpha(x^k)-(1/2)\|v^k\|^2 \right),\qquad k=0,1\ldots,
\]
or, equivalently,
\begin{equation}\label{eq:prop3}
\varphi(F(x^{k+1}))-\varphi(F(x^k))\leq \beta\left((1-\sigma)t_k\alpha(x^k)-(1/2)t_k\|v^k\|^2 \right),\qquad k=0,1\ldots.
\end{equation}
Adding the last inequality from $k=0$ to $n$ and taking into account that $|\alpha(x^k)|=-\alpha(x^k)$, we have
\[
\varphi(F(p^{n+1}))-\varphi(F(p^{0}))\leq -\beta\sum_{k=0}^{n}\left[(1-\sigma)t_k|\alpha(x^k)|+(1/2)t_k\|v^k\|^2 \right].
\]
Thus, because $\beta\in(0,1)$ and $\varphi(F(\bar{x}))\leq \varphi(F(p^{n+1}))$ (see  (\ref{eq:2012-10})), from the last inequality, we get 
\[
\sum_{k=0}^{n}\left[(1-\sigma)t_k|\alpha(x^k)|+(1/2)t_k\|v^k\|^2 \right] \leq\frac{\varphi(F(p^{0}))-\varphi(F(\bar{x}))}{\beta},\qquad n\geq0.
\]
But this tell us that (recall that $\sigma\in[0,1)$)
\begin{equation}\label{eq:2012-1000}
\sum_{k=0}^{+\infty}t_k|\alpha(x^k)|<+\infty\qquad\mbox{and}\qquad  \sum_{k=0}^{+\infty}t_k\|v^k\|^2 <+\infty,
\end{equation}
from which follows the second part of the item $ii$. The first part of the item $ii$ follows from last inequality in (\ref{eq:2012-1000}) together with the fact that $t_k\in(0,1]$.

We assume initially that $\bar{x}$ is an accumulation point of the sequence $\{x^k\}$ and that $\{x^{k_s}\}$ is a subsequence of $\{x^k\}$ converging to $\bar{x}$. From Lemma~\ref{lemma2012-1}, we may conclude that $\{v(x^{k_s})\}$ and $\{\alpha_{x^{k_s}}\}$ converge, respectively, to $v(\bar{x})$ and $\alpha_{\bar{x}}$. In particular, from Lemma~\ref{lemma2012-3}, it follows that $\{v^{k_s}\}$ is bounded and, hence, has a convergent subsequence. Moreover, the sequence $\{t_k\}\subset ]0,1]$ also has an accumulation point $\bar{t}\in [0,1]$. We assume, without loss of generality, that $\{t_{k_s}\}$ converges to $\bar{t}$ and $\{v^{k_s}\}$ converges to some $\bar{v}$. From the equality (\ref{eq:2012-100}), it follows that
\begin{equation}\label{eq:2012-101}
\lim_{s\to+\infty}t_{k_s}\|v^{k_s}\|^2=0.
\end{equation}
We have two possibilities to consider:
\begin{itemize}
	\item [{\bf a)}] $\bar{t}>0$;
	\item [{\bf b)}] $\bar{t}=0$.
\end{itemize}
Assume that item ${\bf a}$ holds. Then, from (\ref{eq:2012-101}), it follows that $\bar{v}=0$.
On the other hand, from the Definition \ref{def:2012} of $v^k$,  we obtain
\[
\max_{1\leq i\leq m}\langle\nabla f_i(x^{k_s}),v^{k_s} \rangle+1/2\|v^{k_s}\|^2\leq(1-\sigma)\alpha(x^{k_s}), \qquad s=0,1\ldots. 
\]
Letting $s$ go to $+\infty$ in above inequality, it follows that $\bar{v}=0$ is a $\sigma$-approximation steepest descent method for $F$ at $\bar{x}$ and, from the item $i$ of Lemma~\ref{lemma2012}, we conclude that $\bar{x}$ is a critical Pareto point of $F$.

Now, assume that item ${\bf b}$ holds true. Since $v^{k_s}$ is a $\sigma$-approximation steepest descent method for $F$ at $x^{k_s}$ and $\{x^{k_s}\}$ is not a critical Pareto point, we have
\[
\max_{i\in I}\langle\nabla f_i(x^{k_s}),v^{k_s}\rangle\leq \max_{i\in I}\langle\nabla f_i(x^{k_s}),v^{k_s}\rangle+(1/2)\|v^{k_s}\|^2<(1-\sigma)\alpha(x^{k_s})<0,
\]
where the last inequality is a consequence of the item $iii$ of Lemma \ref{dir:md1}. Hence, letting $s$ go to $+\infty$ in the last inequalities and using that $\{v^{k_s}\}$ converges to $\bar{v}$, we obtain
\begin{equation}\label{eq:conv1}
\max_{i\in I}\langle\nabla f_i(\bar{x}),v(\bar{x})\rangle\leq (1-\sigma)\alpha(\bar{x})\leq0.
\end{equation}
Take $r\in\mathbb{N}$. Since $\{t_{k_s}\}$ converges to $\bar{t}=0$, we conclude that if $s$ is large enough,
\[
t_{k_s}<2^{-r}.
\]
From (\ref{eq:sl}) this means that the Armijo condition (\ref{eq:dir:md3}) is not satisfied for $t=2^{-r}$, i.e.,
\[
F(x^k+2^{-j}v^{k_s})\npreceq F(x^{k_s})+\beta 2^{-r}J F(x^{k_s})v^{k_s}, 
\]
which  means that there exists at least one $i_0\in I$ such that 
\[
f_{i_0}(x^{k_s}+2^{-r}v^{k_s})> f_{i_0}(x^{k_s})+\beta 2^{-r}\langle\nabla f_{i_0}(x^{k_s}),v^{k_s}\rangle.
\]
Letting $s$ go to $+\infty$ in the above inequality, taking into account that $\nabla f_{i_0}$ and $\exp$ are  continuous  and using that $\{v^{k_s}\}$ converges to $\bar{v}$, we obtain
\[
f_{i_0}(\bar{x}+2^{-r}v(\bar{x}))\geq f_{i_0}(\bar{x})+\beta 2^{-r}\langle\nabla f_{i_0}(\bar{x}),v(\bar{x})\rangle.
\]
The last inequality is equivalent to
\[
\frac{f_{i_0}(\bar{x}+2^{-r}v(\bar{x}))-f_{i_0}(\bar{x})}{2^{-r}}\geq \beta\langle\nabla f_{i_0}(\bar{x}),v(\bar{x})\rangle,
\]
which, letting $r$ go to $+\infty$ and assuming that $0<\beta<1$, yields $\langle\nabla f_{i_0}(\bar{x}),v(\bar{x})\rangle \geq 0$. Hence, 
\[
\max_{i\in I}\langle\nabla f_i(\bar{x}),v(\bar{x})\rangle \geq 0.
\]
Combining the last inequality with (\ref{eq:conv1}) and taking into account that $\sigma\in[0,1)$, we have
\[
\alpha(\bar{x})=0.
\]  
Therefore, from the item $iii$ of Lemma \ref{dir:md1} it follows that $\bar{x}$ is a critical Pareto point of $F$, and the proof is concluded.
\end{proof}
\begin{remark}
If the sequence $\{x^k\}$ begins in a bounded level set, for example, if 
\[
L_F(F(p_0)):=\{ x\in {\mathbb R}^n: F(x)\preceq F(p_0)\},
\]
is a bounded set, then, since $F$ is a continuous function, $L_F(F(p_0))$ is a compact set. So,  item $i$ of \emph{Theorem~\ref{resconv1}} implies that  $\{x^k\}\subset L_F(F(p_0))$ and consequently  $\{x^k\}$ is bounded. In particular,  $\{x^k\}$ has at least one accumulation point.
\end{remark}

\subsection{Full Convergence}\label{S:4.2}
In this section, under the quasi-convexity assumption on $F$, full convergence of the steepest descent method is obtained.

\begin{definition}
Let $H:{\mathbb R}^n\to\mathbb{R}^m$ be a vectorial function.
\begin{itemize}
\item [i)] $H$ is called convex  iff for every $x,y\in {\mathbb R}^n$, the following holds:
\[
H((1-t)x+ty)\preceq (1-t)H(x)+tH(y),\qquad t\in [0,1];
\]
\item [ii)] $H$ is called quasi-convex iff for every $x,y\in {\mathbb R}^n$, the following holds:
\[
H((1-t)x+ty)\preceq \max\{H(x),H(y)\},\qquad t\in [0,1],
\]
where the maximum is considered  coordinate by coordinate;
\item [iii)] $H$ is called pseudo-convex iff $H$ is differentiable and, for every $x,y\in {\mathbb R}^n$, the following holds:
\[
JH(x)(y-x)\not\prec0\quad \Rightarrow\quad H(y)\not\prec H(x).
\]
\end{itemize}
\end{definition}
\begin{remark}\label{re:conv1}
For the two first above definitions see \emph{Definition $6.2$} and \emph{Corollary~$6.6$} of \emph{\cite{Luc1989}}, pages $29$ and $31$, respectively. For the third definition see Definition $9.2.3$ of \cite{yang2002}, page $274$. Note that $H$ is convex (resp. quasi-convex) iff, $H$ is componentwise convex (resp. quasi-convex) . On the other hand, $H$ componentwise pseudo-convex  is a sufficient condition , but not necessary for $H$ to be pseudo-convex; see Theorem $9.2.3$ of \cite{yang2002}, page $274$ and Remark~\ref{remark:2012-377}.  It is immediate from above definitions that if $H$ is convex then it is quasi-convex (the reciprocal is clearly false). If $H$ is differentiable, convexity of $H$ implies that for every $x,y\in {\mathbb R}^n$,
\begin{equation}\label{eq:2012-360}
JH(x)(y-x)\preceq H(y)-H(x),
\end{equation}
from which we may conclude that $H$ is pseudo-convex. It is easy to obtain an example showing that the reciprocal is false.
\end{remark}

Next proposition provides a characterization for differentiable quasi-convex functions. 
\begin{proposition}\label{prop:qc10}
Let $H:{\mathbb R}^n\to\mathbb{R}^m$ be a differentiable function. Then, $H$ is a quasi-convex function if, only if, for every $x,y\in {\mathbb R}^n$, it holds
\[
H(y)\prec H(x)\quad \Rightarrow\quad JH(x)(y-x)\preceq 0.
\]
\end{proposition}
\begin{proof}
Let us assume that, for every pair of points $x,y\in {\mathbb R}^n$, it holds
\begin{equation}\label{eq:2012-380}
H(y)\prec H(x)\quad \Rightarrow\quad JH(x)(x-y)\preceq 0.
\end{equation}
Take $\tilde{x},\tilde{y}\in {\mathbb R}^n$ and assume that holds
\[
H(\tilde{y})\prec H((1-t)\tilde{x}+t\tilde{y}),\qquad t\in [0,1).
\]
Using (\ref{eq:2012-380}) with $y=\tilde{y}$ and $x=(1-t)\tilde{x}+t\tilde{y}$, we obtain $(1-t)JH((1-t)\tilde{x}+t\tilde{y})(\tilde{x}-\tilde{y})\preceq 0$, which implies
\[
\frac{d}{dt}h_i((1-t)\tilde{x}+t\tilde{y})=\langle \nabla h_i((1-t)\tilde{x}+t\tilde{y}),\tilde{y}-\tilde{x}\rangle\leq 0, \qquad i\in \{1,\ldots, m\},
\]
where $h_1,\ldots,h_m$ represent the coordinate functions of $H$. But this implies that 
\[
h_i((1-t)\tilde{x}+t\tilde{y})\leq h_i(\tilde{x}),\quad i\in\{1,\ldots,m\},
\]
and, hence, that 
\[
H((1-t)\tilde{x}+t\tilde{y})\preceq H(\tilde{p})=\max\{H(\tilde{x}),H(\tilde{y})\},
\]
which proves the first part of the proposition. The proof of the second part follows immediately from the definition of quasi-convexity combined with differentiability of $H$; see \cite{Bento2012} for more details. 
\end{proof}

From the previous proposition follows immediately that pseudo-convex functions are quasi-convex. The reciprocal is naturally false. Next proposition provides a sufficient condition for a differentiable quasi-convex function to be pseudo-convex. 

\begin{definition}
A point $x^*\in {\mathbb R}^n$ is a weak  optimal Pareto point of $F$ iff there is no $x\in {\mathbb R}^n$ with $F(x)\prec F(x^*)$.
\end{definition}

\begin{proposition}
Let $H:{\mathbb R}^n\to\mathbb{R}^m$ be a differentiable quasi-convex function. If each critical Pareto point of $H$ is a weak Pareto optimal point, then $H$ is a pseudo-convex function.
\end{proposition}
\begin{proof}
Take $y\in\mathbb{R}^n$. Since that, by hypothesis, each critical Pareto point is an optimal weak Pareto, if $y$ is critical Pareto we have nothing to do. Let us suppose that $y$ is not a critical Pareto point. Then, there exists $v\in \mathbb{R}^n$ such that
\begin{equation}\label{eq:2012-700}
JH(y)v\prec 0.
\end{equation}

Let us assume, by contradiction, that $H$ is not pseudo-convex. In this case, there exists $\tilde{x}\in\mathbb{R}^n$ such that $H(\tilde{x})\prec H(y)$, with
\begin{equation}\label{eq:2012-750}
JH(y)(\tilde{x}-y)\nprec 0.
\end{equation}
From (\ref{eq:2012-700}) and (\ref{eq:2012-750}), it follows that
\begin{equation}\label{eq:2012-790}
JH(y)(\tilde{x}-y)-\beta JH(y)v\npreceq 0, \qquad \beta>0.
\end{equation}
Now, since $H(\tilde{x})\prec H(y)$, from the continuity of $H$ there exists $\delta>0$ such that $H(z)\prec H(y)$ for all $z\in B(\tilde{x},\delta)$ (ball with center in $\tilde{x}$ and ray $\delta$). In particular,
$H(\tilde{x}-(\delta/2)(v/\|v\|))\prec H(y)$ and, because $H$ is quasi-convex, we obtain
\[
JH(y)\left(\tilde{x}-(\delta/2)(v/\|v\|)-y\right)\preceq 0.
\] 
But this tell us that with $\beta=\delta/(2\|v\|)$, we have
\[
JH(y)(\tilde{x}-y)-\beta JH(y)v\preceq 0,
\]
which is a contradiction with (\ref{eq:2012-790}), and the resulted is proved.  
\end{proof}

\begin{remark}\label{remark:2012-377}
Consider the following vectorial function $H:\mathbb{R}\to\mathbb{R}^2$ given by $H(t)=(t,-t^3/3)$. Note that $H$ is not  compenentwise pseudo-convex because $h_2(t):=-t^3/3$ is not pseudo-convex. However, since $H$ is quasi-convex and each critical Pareto point of $H$ is weak Pareto optimal point for $H$, from last proposition, it follows that $H$ is pseudo-convex.  
\end{remark}

We know that criticality is a necessary, but not sufficient,  condition for optimality. In \cite{Bento2012} the authors proved that, under convexity  of the vectorial function $F$, criticality is equivalent to the weak optimality. Next we prove that the equivalence still happens if $F$ is just pseudo-convex.

\begin{proposition}\label{prop:optm1}
Let $H:{\mathbb R}^n\to {\mathbb R}^m$  be a pseudo-convex function. Then, $x\in {\mathbb R}^n$ is a critical Pareto point of $H$, i.e., 
\[
\mbox{Im}(\nabla H(x))\cap (-{\mathbb R}^{m}_{++})=\emptyset,
\]
iff  $x$ is a weak  optimal Pareto point of $H$.
\end{proposition}
\begin{proof}
Let us suppose that $x$ is a critical Pareto point of $H$. Assume by contradiction that $x$ be not a weak optimal Pareto  point of $H$, i.e., that there exists $\tilde{x}\in {\mathbb R}^n$ such that
\begin{equation}\label{eq:cdiff2}
H(\tilde{x})\prec H(x).
\end{equation}
As $H$ is pseudo-convex, then (\ref{eq:cdiff2}) implies that
\[
JH(x)(\tilde{x}-x)\prec 0.
\]
But this contradicts the fact of $x$ being a critical Pareto point of $H$, and the first part is concluded.  The second part is a simple consequence of the fact that $F$ is differentiable with the definitions of critical Pareto point and weak optimal Pareto point. For more details, see \cite{Bento2012}.
\end{proof}

\begin{definition}\label{Fejconv01}
A sequence $\{z^k\}\subset M$ is  quasi-Fej\'er convergent to a  nonempty set $U$ iff, for all $z\in U$, there exists a sequence $\{\epsilon_k\}\subset\mathbb{R}_+$ such that
\[
\sum_{k=0}^{+\infty}\epsilon_k<+\infty, \qquad \|z^{k+1}-z\|^2\leq \|z^k-z\|^2+\epsilon_k, \qquad k=0,1,\ldots.
\]
\end{definition}
In next lemma we recall the theorem known as quasi-Fej\'er convergence.
\begin{lemma}\label{Fejconv1}
Let $U\subset \mathbb{R}^n$ be a nonempty set and $\{z^k\}\subset \mathbb{R}^n$ a quasi-Fej\'er convergent sequence. Then, $\{z^k\}$ is bounded. Moreover, if an accumulation point $\bar{z}$ of $\{z^k\}$ belongs to $U$, then the whole sequence $\{z^k\}$ converges to $\bar{z}$ as $k$ goes to $+\infty$.
\end{lemma}
\begin{proof}
See Burachik et al. \cite{Bur1995}.
\end{proof}
Consider the following set 
\begin{equation}\label{eq:qc1}
U:=\{x\in {\mathbb R}^n: F(x)\preceq F(x^k),\;\; k=0,1, \ldots\}.
\end{equation}
In general, the above set may be an empty set. To guarantee that $U$ is nonempty, an additional assumption on the sequence $\{x^k\}$ is needed. In the next remark we give such a condition.
\begin{remark}\label{remark:2012-500}
If  the sequence $\{x^k\}$ has an accumulation point, then  $U$ is nonempty. Indeed, let $\bar{x}$ be an accumulation point of the sequence $\{x^k\}$. Then, there exists a subsequence $\{p^{k_j}\}$ of $\{x^k\}$ which converges to $\bar{x}$. Since $F$ is continuous $\{F(x^k)\}$ has $F(\bar{x})$ as an accumulation point. Hence, using $\{F(x^k)\}$ as a decreasing sequence \emph{(}see \emph{item i} of \emph{Theorem~\ref{resconv1}}\emph{)} the usual arguments easily show that the whole sequence  $\{F(x^k)\}$ converges to $F(\bar{x})$  and the following relation holds
\[
F(\bar{x})\preceq F(x^k),\qquad  k=0,1, \ldots,
\]
which implies that $\bar{x}\in U$, i.e., $U\neq\emptyset$.
\end{remark}

{\bf Assumption 1.} Each $v^k$ of the sequence $\{v^k\}$ is a scalarization compatible, i.e., exists a sequence $\{w^k\}\subset\mbox{conv} S$ such that
\[
v^k=-JF(x^k)^tw^k,\qquad k=0,1,\ldots.
\]
As was observed in Section~\ref{sec3}, this assumption holds if $v^k=v(x^k)$, i.e., if $v^k$ is the exact steepest descent direction at $x^k$. We observe that the Assumption 1 also was used in \cite{Benar2005} for proving the full convergence of the sequence generated for  the Method in the case that $F$ is convex. From now on, we will assume that the Assumption 1 holds true. 

In next lemma we present the main result of this section. It is fundamental to the proof of the global convergence result of the sequence $\{x^k\}$.
\begin{lemma}\label{Fejconv2}
Suppose that $F$ is quasi-convex and $U$, defined in (\ref{eq:qc1}), is nonempty. Then, for all $\tilde{x}\in U$, the following inequality is true:
\[
\|x^{k+1}-\tilde{x}\|^2\leq\|x^k-\tilde{x}\|^2+t_k^2\|v^k\|^2.
\]  
\end{lemma}
\begin{proof}
Consider the hinge $\left(\overline{x^k\tilde{x}},\overline{x^{k}x^{k+1}},\alpha\right)$, where $\overline{x^k\tilde{x}}$ is the segment joining $x^k$ to $\tilde{x}$; 
$\overline{x^kx^{k+1}}$ is the segment joining $x^k$ to $x^{k+1}$ and $\alpha=\angle(\tilde{x}-x^k,v^k)$. By the law of cosines, we have
\[
\|x^{k+1}-\tilde{x}\|^2= \|x^k-\tilde{x}\|^2+t_k^2\|v^k\|^2-2t_k\|x^k-\tilde{x}\|\|v^k\|\cos\alpha,\qquad k=0,1,\ldots.
\]
Thus, taking into account that $\cos(\pi-\alpha)=-\cos\alpha$ and $\langle-v^k,\tilde{x}-x^k\rangle=\|v^k\|\|x^k-\tilde{x}\|\cos(\pi-\alpha)$, the above equality becomes
\[
\|x^{k+1}-\tilde{x}\|^2= \|x^k-\tilde{x}\|^2+t_k^2\|v^k\|^2+2t_k\langle -v^k,\tilde{x}-x^k\rangle,\qquad k=0,1,\ldots.
\]
On the other hand, from Assumption 1, there exists $w^k\in\mbox{conv} S$ such that 
\[
v^{k}=-JF(x^k)^tw^k,\qquad k=0,1,\ldots.
\]
Hence, the last equality yields 
\[
\|x^{k+1}-\tilde{x}\|= \|x^k-\tilde{x}\|^2+t_k^2\|v^k\|^2+2t_k\langle JF(x^k)^tw^k,\tilde{x}-x^k\rangle,\quad k=0,1,\ldots,
\]
from which, we obtain
\begin{equation}\label{eq:contv10}
\|x^{k+1}-\tilde{x}\|^2= \|x^k-\tilde{x}\|^2+t_k^2\|v^k\|^2+2t_k\langle w^k,JF(x^k)(\tilde{x}-x^k)\rangle,\quad k=0,1,\ldots.
\end{equation}
Since $F$ is quasi-convex and $\tilde{x}\in U$, from Proposition \ref{prop:qc10} with $H=F$, $x=x^k$ and $y=\tilde{x}$, we have
\[
J F(x^k)(\tilde{x}-x^k)\preceq 0, \qquad k=0,1,\ldots.
\]
So, because $w^k\in\mbox{conv} S$, we get
\begin{equation}\label{eq:contv11}
\langle w^k,JF(x^k)(\tilde{x}-x^k)\rangle\leq 0,\quad k=0,1,\ldots.
\end{equation}
Therefore, by combining (\ref{eq:contv10}) with (\ref{eq:contv11}), the lemma proceeds. 
\end{proof}

\begin{proposition} \label{pr:cf}
If $F$ is quasi-convex, $\mathbb{R}^n$ has non-negative curvature and $U$, defined in (\ref{eq:qc1}), is a nonempty set, then the sequence $\{x^k\}$ is quasi-Fej\'er convergent to $U$.
\end{proposition}
\begin{proof}
The resulted follows from the item $ii$ of Theorem~\ref{resconv1} and Lemma \ref{Fejconv2} combined with Definition \ref{Fejconv01}.
\end{proof}
\begin{theorem}
Suppose that $F$ is quasi-convex, and $U$, as defined in (\ref{eq:qc1}), is a nonempty set. Then, the sequence $\{x^k\}$ converges to a critical Pareto point of $F$.
\end{theorem}
\begin{proof}
From  Proposition~\ref{pr:cf}, $\{x^k\}$ is Fej\'er convergent to $U$. Thus, Lemma \ref{Fejconv1} guarantees that $\{x^k\}$ is bounded and, hence, has an accumulation point $\bar{x}\in {\mathbb R}^n$. Thus, from Remark~\ref{remark:2012-500}, we conclude that $\bar{x}\in U$ and, hence, that the whole sequence $\{x^k\}$ converges to $\bar{x}$ as $k$ goes to $+\infty$ (see Lemma \ref{Fejconv1}). The conclusion of the proof is a consequence of item $iii$ of Theorem~ \ref{resconv1}.
\end{proof}
\begin{corollary}\label{res:fconv100}
If $F$ is pseudo-convex, $\mathbb{R}^n$ has non-negative curvature and $U$, as defined in (\ref{eq:qc1}), is a nonempty set, then the sequence $\{x^k\}$ converges to a weak optimal Pareto point of $F$.
\end{corollary}
\begin{proof}
Since $F$ is pseudo-convex, and in particular quasi-convex, the corollary is a consequence of the previous theorem and Proposition \ref{prop:optm1}.
\end{proof}

\section{Variational Rationality: Inexact Proximal Algorithms as Self Regulation Problems}\label{sec5}

In this section, we consider an endless unsatisfied man, who, instead to renounce, aspires, and partially
satisfice, using worthwhile changes.

\subsection{Variational rationality}

\noindent{\bf  1) The course between unsatisfied needs, aspirations, and
satisfaction levels}

Variational rationality (Soubeyran \cite{Soubeyran2009, Soubeyran2010, Soubeyran2012a, Soubeyran2012b}) is a
purposive and dynamic approach of behaviors. It modelizes the course pursuit
between desired ends and feasible means. It is a theory of the endless
unsatisfied man, who, given a lot of unsatisfied needs, both renounces to
satisfy some of them and aspires to satisfice some others. Let us summarize
some of the main points of this conative approach of behaviors, based on
cognition (knowledges), motivation (desires) and affect (feelings and
emotions),
\begin{itemize}
\item [i)] the agent, focusing his attention on the unsatisfied needs he has chosen to satisfice, considers desired ends. He forms aspirations (distal goals).
Setting aspirations is a way to know what he really wants among all his
wishes, without considering if they are realistic or not.

\item [ii)] then, the agent starts to consider feasible means (defined as the means
he must find, build, gather and learn how to use);

\item [iii)] given the difficulty to gather such feasible means, the agent chooses
to partially satisfice his aspirations;
\end{itemize}
Then, the agent self regulates all his goal oriented activities:
\begin{itemize}
\item [iv)] goal setting, setting proximal goals is a way for him to divide the
difficulty, to better know what he can really do. This allows him to balance
between ``desired enough" ends and ``feasible enough" means;

\item [v)] goal striving represents the path (way of doing, strategy) the agent
chooses to follow and the obstacles he must overcome to attain his
successive proximal goals and partially satisfice;

\item [vi)] goal pursuit is the revision of his goals, using feedbacks coming from
successes and failures.
\end{itemize}

This variational approach is progressive (adaptive). The step by step joint
formation of distal (global) and proximal (local) goals and related actions
is a process including a lot of interactions, tatonnements,
adjustments,\ldots, driven by inexact perceptions, evaluations and judgments.

Among several variational principles, three of them are worth mentioning in
the present paper,

- the ``satisficing with not too much sacrificing" principe and the
``worthwhile to change principle" (Soubeyran \cite{Soubeyran2009, Soubeyran2010}) \ 

- the ``tension reduction-tension production" principle (Soubeyran \cite{Soubeyran2012a, Soubeyran2012b}).

\noindent{\bf 2) Unsatisfied needs, aspirations, satisfaction and satisfying
levels, and aspiration gaps}

In the specific case of this paper let us modelize the main motivational
concepts of the variational approach of Soubeyran \cite{Soubeyran2009, Soubeyran2010}. They include,

a) \textbf{The map of unsatisfied needs (the needs system).}   An agent has two ways to perceive, judge and estimate a situation, either
in term of unsatisfaction or in term of satisfaction. Usually an agent deals with a lot of unsatisfied needs which depend of his present
situation $x\in X$. Let $I=\left\{ 1,2,\ldots,m\right\} $ be the list of
different potential needs. The perceived unsatisfied  needs functions are $0\leq n_{i}(x)<+\infty$ be the strength
of each perceived need $i\in I$ for each situation $x\in X.$ Let $N:X\to\mathbb{R}^m$, given by 
\[
N(x)=(n_{1}(x),n_{2}(x),\ldots, n_{i}(x),\ldots,n_{m}(x)),
\]
be the map of unsatisfied needs in this situation. These needs can be rather vague and
abstract. Hull \cite{Hull1935} and Murray \cite{Murray1938} give an extensive list of different
needs.

b) \textbf{The map of aspiration gaps.} As soon as the agent chooses to do not renounce to satisfy,
at least partially, all these unsatisfied needs, they become, in this
present situation, aspiration gaps $a_{i}(x), i\in I$, although in general, the perceived aspiration gaps $a_{i}(x)\geq 0,i\in I$ are lower than perceived unsatisfied needs: $0\leq a_{i}(x)\leq $ $n_{i}(x),i\in I$ because agents usually aspire to fill no more than their unsatisfied needs. 

c) \textbf{The map of aspiration levels (desirable ends, or the distal goal system).} Let us denote by $\overline{g}_{i}(x)$, $i\in I$, the aspiration levels for $x\in X$. They represent still vague, abstract, and non committed higher order
goals (visions, ideals, aspirations, fantasies, dreams, wishes, hopes,
wants, and desires). These aspirations levels represent
desirable (but perhaps irrealistic) ends. Lewin \cite{Lewin1951} defines
aspiration levels as desirable ends, some being irrealistic
in a near future, and others not.

d) \textbf{The map of satisfaction levels (the experienced utility system).}
\textbf{\ \ }Most of the time unsatisfied needs are partially satisfied.
Let $G:X\to\mathbb{R}^m$, $G(x)=(g_{1}(x),g_{2}(x),..,g_{i}(x),..,g_{m}(x))$, be the map
of present satisfaction levels (or outcomes) in the present situation $x$,
where $g_{i}(x)\leq \overline{g}_{i}(x)$, $i\in I$, i.e, the levels at which all needs
are partially satisfied. 

e) \textbf{The map of discrepancies (the drive system).} The
differences between aspiration levels and satisfaction levels define more
precisely aspiration gaps $a_{i}(x)=\overline{g}_{i}(x)-g_{i}(x)\geq 0,i\in I$ which are non negative. We will assume in this paper that aspirations gaps are
equals to unsatisfied needs, i.e, they represent the discrepancies
\[
f_{i}(x)=n_{i}(x)=a_{i}(x)=\overline{g}_{i}(x)-g_{i}(x)\geq 0,i\in I,\qquad
x\in X,
\]
because, usually, agents aspire to satisfy their perceived unsatisfied
needs, even if, in a second stage, they have not the intention to satisfy
all of them. The crude perception of these gaps generates feelings and
emotions, the so called drives, Hull \cite{Hull1935}.

We consider here satisfaction levels $g_{i}(x),\,i\in I$, instead of
discrepancies. Moreover, for simplification, we consider all aspiration
levels as constant, i.e, $\overline{g}_{i}(x)=\overline{g}_{i}<+\infty ,$ $%
i\in I$, $x\in X.$ Then, 
\[
f_{i}(x)=\overline{g}_{i}-g_{i}(x)\geq 0,\quad \mbox{and}\quad
g_{x}(v)=-f_{x}(v),\qquad i\in I,\;x\in X. 
\]

%e) \textbf{The map of satisficing levels (the proximal goals system).} Because aspirations levels $\overline{g}_{i}>g_{i}(x)$, $i\in I$ can be very desirable but irrealistic ends, the agent considers satisficing levels $%
%\widetilde{g}_{i}(x),i\in I$ which are more realistic, but less desirable
%ends, higher than the present satisfaction levels $g_{i}(x),i\in I$, but
%lower than the aspiration levels $\overline{g}_{i},i\in I,$ i.e $g_{i}(x)<$
%$\widetilde{g}_{i}(x)\leq \overline{g}_{i},i\in I.$ Simon \cite{Simon1955} defines
%such ``satisficing levels" and names them ``aspiration levels" (which later
%creates a confusion in terminology among his followers). Bandura \cite{Bandura1977a, Bandura1977b} defines such proximal goals.

The main problem is to know how the agent sets all these levels, step by
step (progressively).

\noindent{\bf 3) Feasible means and the ``goal system"}

A ``goal system" (Kruglanski et al.\cite{Kruglanski2002}) comprises i) a cognitive
network of mental representations of goals and means which are structurally
interconnected by several more or less strong cognitive links, ii) in the
short run a subset of limited available resources (physiological, material,
mental, social means) because means are scarce and difficult to obtain, iii)
an allocation process where goals compete for the use of these limited
available resources, iv) a motivational process of goal setting, goal
commitment, goal striving and goal pursuit (using affective feedback
engendered in response to success and failure outcomes, goal revision
including persistence of pursuit, means substitution and the management of
goal-conflict). In our specific case the goal system is the satisfaction map 
\[
X\ni x\longmapsto G(x)\in R^{m}.
\]
Available means are identied to the situation $x\in X$. These means can represent actions, resources and
capabilities ``to be able to do" them (see Soubeyran \cite{Soubeyran2009, Soubeyran2010}). The fact
that outcomes compete for restricted means can be modelized as follows: we
decompose the given $x$ into the sum $\ x=x_{1}+x_{2}+...+x_{i}+.....+x_{k}$
where $x_{i}\in X$ is the bundle of means allocated to goal $i$ and $%
g_{i}=g_{i}(x_{i})$ is the level of satisfaction of this objective.

\subsection{The proximal ``satisficing-but not too much sacrificing" principle%
}

\noindent\textbf{The local evaluation of marginal satisfaction levels of change}.
Starting from the situation $x\in X$, let $v\in X$ be a direction of change$%
, $ $t>0$ be the intensity of change, $u=y-x=$ $tv$ be the change and $%
G_{x}(v)=JG(x)v$ be the vector of marginal satisfaction levels of change.
The related differents needs may have different degree of importance and
urgency and, each step, the agent must weight each of them to define
priorities. This task (solving trade off) is not easy and must be done
progressively. Define $g_{x}:\mathbb{R}^n\to\mathbb{R}$, given by
\[
g_x(v):=\min_{i\in I} \langle\nabla g_i (x),v\rangle=\min_{i\in I} (JG(x)v)_{i},i\in I,\qquad i\in I,
\]
the marginal satisfaction function. It represents the minimum of the
different marginal satisfaction levels $(JG(x)v)_{i},i\in I$. The
consideration of this marginal satisfaction function avoids to choose
weights for each marginal satisfaction level and to have to adapt them each
step.

\noindent\textbf{Taking care of exploration costs. }Let situations like $x\in X$
represents means which generate the vector of satisfaction levels $G(x)$. \
The agent, in situation $x\in X$, considers (explores) new situations $y=x+tv,t>0,v\in X.$ This global exploration process is costly. Let us define
the local consideration costs (search costs, exploration costs) $%
c_{x}(v)=(1/2)\left\Vert v\right\Vert ^{2}\geq 0.$ The choice of a quadratic
function modelizes the case where local exploration is not too coslty:
consideration costs $c_{x}(v)\geq 0$ are large ``in the large" and small ``in
the small". Notice that, while the agent considers feasible directions of
change over the whole state space, he takes care of consideration costs (a
local aspect).

\noindent\textbf{The local search of directions of aspirations. }In general\textbf{\ }%
(Soubeyran \cite{Soubeyran2009, Soubeyran2010}) \textbf{\ }the\textbf{\ }proximal payoff balances
desired ends and feasible means. In this paper the proximal payoff $%
l_{x}(v)=g_{x}(v)-c_{x}(v)$ balances the marginal satisfaction levels
and the costs to consider them (local exploration costs). Since 
\[
l_{x}(v)\geq
0\Longleftrightarrow g_{x}(v)\geq c_{x}(v),
\]
we will say that it is ``worthwhile to explore in direction $v,$ starting from $x$", because the
marginal satisfaction level $g_{x}(v)$ in this direction is higher than the
costs $c_{x}(v)$ to be able to consider them. For each $x\in X$ consider the
local search proximal problem: find a direction of change $v(x)\in X$ such
that $l_{x}(v(x))=$ $\sup \left\{ l_{x}(v),v\in X\right\} .$ Let $\overline{l}%
_{x}=\sup \left\{ l_{x}(v),v\in X\right\} $ be the optimal proximal payoff
function at $x$ and $v(x)=\arg \max \left\{ l_{x}(v),v\in X\right\} \in X$
be the unique optimal direction of change, starting from $x.$ Then, $%
\overline{l}_{x}=l_{x}(v(x)).$ From Lemma \ref{dir:md1} and Lemma \ref{lemma2012-1}, it follows that

1) If \ $x\in X$ is Pareto critical, then $v(x)=0\in X$ and $\overline{l}%
_{x}=0.$

2) If \ $x\in X$ is not Pareto critical, then $\overline{l}_{x}>0$ and $%
g_{x}(v)\geq $ $g_{x}(v(x))>c_{x}(v(x)),i\in I.$

3) The mappings  $X\ni x\mapsto v(x) \in X$ and $X\ni x\mapsto 
\overline{l}_{x}\in R$ are continuous.

Starting from $x$ and using variational rationality concepts, the optimal
direction of change $v(x)$ defines the unique direction of aspiration and
the optimal proximal payoff $\overline{l}_{x}=l_{x}(v(x))$ defines the
proximal aspiration level (net of consideration costs). From Lemma \ref{dir:md1} and Lemma \ref{lemma2012-1}, it follows that

1) If \ $x\in X$ is Pareto critical, then the direction of aspiration and
the proximal aspiration level are zero.

2) If \ $x\in X$ is not Pareto critical, then, their is a strictly positive
direction of aspiration and the proximal aspiration level is strictly
positive.

3) direction of aspirations and proximal aspiration levels are continuous

\noindent\textbf{The local determination of a local satisficing direction of change:}  If  $x\in X$ is not Pareto critical,

i) set a local (net) satisficing level of change $\widetilde{l}_{x}=\sigma 
\overline{l}_{x},0<\sigma \leq 1$ which is positive and strictly lower than
the local aspiration level of change $\overline{l}_{x}>0.$ As a ``variational
rationality" concept (Soubeyran \cite{Soubeyran2009, Soubeyran2010}), this local satisficing level
of change $\widetilde{l}_{x}$ is situational dependent (it changes with $x\in X$). Simon \cite{Simon1955}, the father of the satisficing concept, defines an
invariant satisficing level without any reference to an aspiration level of
change $\overline{l}_{x}$.

ii) Then, the agent will try to find a direction $v^{\S }\in X$ such that $%
l_{x}(v^{\S })\geq \widetilde{l}_{x}.$ This means that the satisficing
direction of change $v^{\S }$ ``improves enough" with respect to the
aspiration direction of change $v(x)$, including \ exploration costs. In
variational term such a direction not only satisfices (Simon \cite{Simon1955}) but even
more, it balances satisficing (``improving enough") marginal satisfactions
to change $n_{x}(v)$ with some sacrifices to change $c_{x}(v)$, because the
net satisfaction level $l_{x}(v^{\S })$ is higher than the (net) satisficing
level. This is a local version of the variational ``sacrificing with not too
much sacrificing" principle (Soubeyran \cite{Soubeyran2009, Soubeyran2010}). This is equivalent to
say that the satisficing direction of change $v^{\S }\in X$ is an inexact
solution of the local search proximal problem. In this context proximal
goals are local aspiration levels of change and satisficing levels of change.

\begin{remark} 
In term of variational rationality, the Fliege-Swaiter \cite{Benar2000}
steepest descent method appears to be an ``aspiration driven local search
proximal algorithm". In situation $x\in X$, the distal goal is the
aspiration level of change $\overline{l}_{x}$ and the proximal goal is the
satisficing level $\widetilde{l}_{x}.$
\end{remark}
\subsection{The proximal ``worthwhile to change" principle and goal difficulty%
}

Inertia matters because to be able to change from some situation $x\in X$ to
a new improving situation $y\in X$ is costly. As variational concepts, there
are two kinds of ``costs to be able to change" (Soubeyran \cite{Soubeyran2009, Soubeyran2010}): i)
consideration costs (perception, exploration, search and evaluation
costs \ldots), ii) capability costs to change $C(x,y)$, i.e, the costs to be
able to change (to be able to stop to use old means, to be able to use
again old means, and to be able to imagine, find, build, gather, and learn
how to use new means). Means can be capabilities ( competences, skills),
ingredients and resources. In the present paper consideration costs are $%
c_{x}(v)=(1/2)\left\Vert v\right\Vert ^{2}$ and capability costs to change
are $C(x,y)=K\left[ tJG(x)(y-x)\right] =tJG(x)(y-x),$ with $t>0.$ This
formulation, specific to the present paper, means that costs to change from $%
x$ to $y$ increase with the difficulty to change, modelized here as the
vector of gradients $\Lambda (x,y)=$ $tJG(x)(y-x)$, including the step
length of change $t>0.$ Then, the second variational principle tells us that
it is ``worthwhile to change" from $x$ to $y$ if advantages to change $%
A(x,y)=G(y)-G(x)$ are higher than some proportion, $\beta >0,$ of costs to
change, i.e $G(y)-G(x)\geq \beta C(x,y)$ where $\beta >0$ is a rate of
tolerance which calibrates how the change (transition) $x\curvearrowright y$
is acceptable$.$ More generally (Soubeyran \cite{Soubeyran2009, Soubeyran2010}) it is ``worthwhile
to change" from $x$ to $y$ if motivations to change $M(x,y)=U\left[ A(x,y)%
\right] $ are higher than some proportion, $\beta >0,$ of resistances to
change $R(x,y)=\Lambda \left[ C(x,y)\right] $ where $U(.)$ and $\Lambda (.)$
are the experienced utility and desutility of advantages and costs to change.

The variationel concept of ``worthwhile changes" is related to the famous
Lindblom \cite{Lindblom1959} ``muddling through" economizing principle where agents make
small steps (incremental changes, choosing the step size in our context)
and successive limited comparisons (balancing pro and cons).

\begin{remark} 
Our paper considers quasi-convex (or quasi-concave) payoffs. In
term of variational rationality, this case is very interesting, because it
allows large flat portions which can be very costly to explore (quadratic
exploration costs are ``large in the large"). Hence, in this case,
convergence is a very nice result.
\end{remark}
\subsection{Local exploration traps}

The goal of this paper has been to give conditions of convergence of a path
of change towards a Pareto optimum in a multicriteria optimization setting.
The variational concept of a behavioral trap (Soubeyran \cite{Soubeyran2009, Soubeyran2010})
appears in this context at the local level of the consideration (say
exploration) process. More precisely, we will say that $x^{\ast }\in X$ is a
local exploration trap if 
\[
\overline{l}_{x^{\ast }}=0\Longleftrightarrow
l_{x^{\ast }}(v)=g_{x^{\ast }}(v)-c_{x^{\ast }}(v)\leq 0,\qquad v\in X.
\]
This means that, locally, it is not worthwhile to explore, because, whatever
the direction of change $v\in X,$ marginal advantages to change $n_{x^{\ast
}}(v)$ are lower than local exploration costs to change $c_{x^{\ast }}(v).$
Lemma \ref{dir:md1} shows that if $x^{\ast }\in X$ is Pareto critical, then $x^{\ast
}\in X$ is a local exploration trap.

\section{Final Remarks}
We proved full  convergence of the sequence generated by this inexact method to a critical Pareto point associated to quasi-convex multicriteria optimization problems.  We also show a striking result, i.e, the strong connexion of such an inexact proximal algorithm with the self regulation problem in Psychology. Further researches can be made in this direction.

%% The Appendices part is started with the command \appendix;
%% appendix sections are then done as normal sections
%% \appendix

%% \section{}
%% \label{}

%% References
%%
%% Following citation commands can be used in the body text:
%% Usage of \cite is as follows:
%%   \cite{key}          ==>>  [#]
%%   \cite[chap. 2]{key} ==>>  [#, chap. 2]
%%   \citet{key}         ==>>  Author [#]

%% References with bibTeX database:

\bibliographystyle{model1-num-names}
\bibliography{<your-bib-database>}

%% Authors are advised to submit their bibtex database files. They are
%% requested to list a bibtex style file in the manuscript if they do
%% not want to use model1-num-names.bst.

%% References without bibTeX database:

% \begin{thebibliography}{00}

%% \bibitem must have the following form:
%%   \bibitem{key}...
%%

% \bibitem{}

% \end{thebibliography}

\end{document}